\documentclass[12pt, a4paper, oneside]{amsart}
\usepackage[a4paper, left=2cm, right=2cm, top=2cm, bottom=2cm]{geometry}
\usepackage[english]{babel}
\usepackage[T2A]{fontenc}
\usepackage[utf8]{inputenc} 
\usepackage{amsfonts}
\usepackage{amssymb, amsthm, amscd}
\usepackage{amsmath}
\usepackage{mathtools}
\usepackage{needspace}
\usepackage{etoolbox}
\usepackage{lipsum}
\usepackage{comment}
\usepackage{cmap}
\usepackage[pdftex]{graphicx}
\usepackage[unicode]{hyperref}
\usepackage[matrix,arrow,curve]{xy}
\usepackage[usenames,dvipsnames]{xcolor}
\usepackage{colortbl}
\usepackage{textcomp}
\usepackage{cite}
\usepackage{euscript}

\DeclareFontFamily{U}{mathx}{}
\DeclareFontShape{U}{mathx}{m}{n}{<-> mathx10}{}
\DeclareSymbolFont{mathx}{U}{mathx}{m}{n}
\DeclareMathAccent{\widehat}{0}{mathx}{"70}
\DeclareMathAccent{\widecheck}{0}{mathx}{"71}

\makeatletter
\def\@settitle{\begin{center}
		\baselineskip14\p@\relax
		\bfseries
		\LARGE
		\@title
	\end{center}
}
\makeatother

\newtheorem{theorem}{Theorem}[section]
\newtheorem{proposition}[theorem]{Proposition}
\newtheorem{lemma}[theorem]{Lemma}
\newtheorem{corollary}[theorem]{Corollary}
\newtheorem{definition}[theorem]{Definition}
\newtheorem{mainth}{Theorem}

\theoremstyle{definition}

\newtheorem{remark}{Remark}

\renewcommand{\geq}{\geqslant}
\renewcommand{\leq}{\leqslant}
\newcommand{\R}{\mathbb{R}}
\newcommand{\N}{\mathbb{N}}

\newcommand{\F}{\mathcal{F}}

\newcommand{\K}{\mathcal{K}}
\newcommand{\codim}{\mathrm{codim }}

\newcommand{\CC}{\mathbb{C}}

\DeclareMathOperator{\Ker}{Ker}

\DeclareMathOperator{\Span}{Span}

\newenvironment{enumerate*}%
  {\begin{enumerate}%
    \setlength{\itemsep}{1pt}%
    \setlength{\parskip}{1pt}}%
  {\end{enumerate}}

\makeatletter
\newcommand*{\@old@slash}{}\let\@old@slash\slash
\def\slash{\relax\ifmmode\delimiter"502F30E\mathopen{}\else\@old@slash\fi}
\makeatother

\title{Hereditary completeness for systems of exponentials in weighted $L^2$-spaces}
\author{Andrei V. Semenov}

\thanks{The work was supported by RSF grant 26-11-00157. }


\address{Andrei V. Semenov
Saint Petersburg State University\\
Department of Mathematics and Computer Science\\
Russia, 199178, Saint Petersburg, 14th Line of Vasilievsky Island, 29}

\email{asemenov.spb.56@gmail.com}

\begin{document}
\maketitle
\begin{abstract}
We prove that for every $0 < \alpha< 1$ there exists $\Lambda_{\alpha} \subset \R$ such that the system $\{e^{i\lambda t}\}_{\lambda \in \Lambda_{\alpha}}$ is complete and minimal, but not hereditarily complete in the weighted space $L^2(w)$ for any weight $w$ with $w(t) \asymp (\pi-|t|)^{\alpha}$ on $(-\pi,\pi)$.
\end{abstract}

\section{Introduction}
A system $\{v_n\}_{n \in \N}$ in a separable Hilbert space $H$ with inner product $(-,-)$ is {\it complete} if $\overline{\Span{\{v_n\}_{n \in \N}} }= H$. It is {\it minimal} if $v_m \not\in \overline{\Span{ \{v_n\}_{n \ne m}}}$ for every $m\in \N$. Every complete and minimal system (an {\it exact} system) $\{v_n\}_{n \in \N}$ admits a unique biorthogonal system $\{w_n\}_{n \in \N}$ such that $(v_n, w_m) =\delta_{n,m}$. This allows one to associate a formal Fourier-type series with every element of the space:
$$x \sim \sum_{n \in \N} (x, w_n)v_n \quad \text{for any } x \in H.$$
A natural question arises: under what minimal assumptions can an element be recovered from its Fourier coefficients?

This reconstruction problem is closely related to spectral synthesis. Spectral synthesis asks whether every invariant subspace of a linear operator can be reconstructed from the root vectors it contains. Markus \cite{M} studied spectral synthesis for operators with point spectrum and related it to the completeness of biorthogonal expansions. For further details we refer the reader to \cite{BBB1, BBB2, BBK, M}.

For systems of vectors in a Hilbert space, spectral synthesis takes the form of \textit{hereditary completeness}. Hereditary completeness represents the weakest topological property that still permits a meaningful reconstruction of vectors from their Fourier series (see \cite{BB} and the introduction of \cite{BBB} for an extended discussion). We shall use the following nonstandard definition of this property (see \cite{DN}):

\begin{definition}
An exact system $\{v_n\}_{n \in \N} \subseteq H$ with biorthogonal system $\{w_n\}_{n\in \N}$ is {\it hereditarily complete} if for every partition $\N = A \cup B$ such that $A \cap B = \emptyset$, the mixed system $\{v_n\}_{n\in A} \cup \{w_n\}_{n\in B}$ is complete in $H$.
\end{definition}

From now on we take $H = L^2(-\pi, \pi)$. We say that $\{v_n\}$ is a {\it system of exponentials} if $v_n = e^{i \lambda_n t}$ for some sequence $\{\lambda_n\}_{n \in \N}$. Such systems appear not only in classical Fourier analysis, but also in modern frame theory (see \cite{OU}) and the Beurling--Malliavin theory (see \cite{MP1, MP2}). It had long been an open question whether every exact system of exponentials automatically admits spectral synthesis. 

This question was resolved in the negative by Baranov, Belov, and Borichev (see Theorem 1.3 in \cite{BBB2}): they discovered that there exist exact systems of exponentials in $L^2(-\pi, \pi)$ that fail to be hereditarily complete. Moreover, they showed that the completeness defect of any mixed system is at most one-dimensional: for every partition $\N = A \cup B$ and every complete and minimal system of exponentials $\{v_n\}_{n \in \N}$ we have
$$\codim \text{ } \overline{\Span \left( \{v_n\}_{n\in A} \cup \{w_n\}_{n\in B} \right) } \leq 1,$$
where $\{w_n\}$ is the biorthogonal system (see Theorem 1.1 in \cite{BBB2}). Related spectral-synthesis phenomena have since been studied for reproducing kernels in de Branges spaces (see \cite{BBB1}) and for Gaussian Gabor systems (see \cite{BBB}). In a different direction, Zikkos and Gunatillake proved that certain systems of exponentials with real frequencies are hereditarily complete in their closed linear spans (see \cite{ZG}).

It is natural to ask whether this failure is specific to the unweighted space $L^2(-\pi,\pi)$. Does this phenomenon persist when the weight vanishes at the endpoints? In various applications (see \cite{Is} and \cite{LYu}) one considers weighted spaces $L^2(w)$:
\begin{equation}\label{eq:defL}
L^2(w) = \left\{ f : (-\pi,\pi) \to \CC \colon\int_{-\pi}^{\pi} |f(t)|^2 w(t) \, dt < \infty \right\}.
\end{equation}
We study hereditary completeness of exponentials in $L^2(w)$ for the weights $w$ with polynomial decay at the endpoints.

\section{Main result}
We equip the space $L^2(w)$ from \eqref{eq:defL} with the norm and the inner product
$$\|f\|_w^2 := \int_{-\pi}^{\pi} |f(t)|^2 w(t) \, dt  \quad \text{ and } \quad (f,g)_w := \int_{-\pi}^{\pi} f(t)  \overline{g(t)} \, w(t)\, dt.$$

Throughout the paper we fix $0<\alpha<1$ and put
\[
w_{\alpha}(t)=(\pi-|t|)^{\alpha} \quad \text{for } t\in (-\pi, \pi).
\]

\begin{definition}\label{def:Ka}
Fix $0<\alpha<1$. Define $\K_{\alpha}$ to be the class of all measurable functions $w:(-\pi,\pi)\to(0,+\infty)$ for which there exist constants $0<c_w\leq C_w<\infty$ such that
\[
c_w w_{\alpha}(t)\leq w(t)\leq C_w w_{\alpha}(t) \quad \text{ for a.e. } t\in(-\pi,\pi).
\]
\end{definition}

\begin{remark}
Note that $w \in L^1(-\pi,\pi)$ and $1/w \in L^1(-\pi,\pi)$ for every $w\in\K_{\alpha}$. We first carry out the construction for the weight $w_{\alpha}$. The passage to an arbitrary $w\in\K_{\alpha}$ is then a standard argument developed in Subsection~\ref{sect:eqiv}.
\end{remark}

 Let us state our main result. \par
\begin{mainth}\label{thm:main}
For every $0<\alpha<1$ there exists a discrete set $\Lambda_{\alpha}\subset\R$ such that for every weight $w\in\K_{\alpha}$ the system $\bigl\{e^{i\lambda t}\bigr\}_{\lambda\in\Lambda_{\alpha}}$ is complete and minimal but is not hereditarily complete in $L^2(w)$.
\end{mainth}

We also expect an analogue of \cite[Theorem~1.1]{BBB2} to hold in $L^2(w)$, which leads to the following question. For any $w \in \mathcal{K}_{\alpha}$ and any complete and minimal system of exponentials $\{v_n\}_{n\in \N}$ in $L^2(w)$, let $\{w_n\}_{n \in \N}$ be its biorthogonal system. Consider a partition $\N = A \cup B$ such that $A \cap B = \emptyset$.\par
{\bf Question 1}. {\it Is it true that the completeness defect of the mixed system $\{v_n\}_{n \in A} \cup \{w_n\}_{n \in B}$ is  finite?}\par
The paper is organized as follows. In Section~\ref{sect:Prelim}, we develop the technique needed to pass from $w_{\alpha}$ to an arbitrary weight $w\in\K_{\alpha}$. In Section~\ref{sect:Proof}, we construct the set $\Lambda_{\alpha}$ for the weight $w_{\alpha}$. Section~\ref{sect:Tech} contains the quantitative estimates and completes the proof of Theorem~\ref{thm:main}.

\subsection{Notations}
We write $U(z)\lesssim V(z)$ if $U(z)\leq C V(z)$ for some constant $C>0$ and all $z$ in the range under consideration. We write $U(z)\asymp V(z)$ if $U(z)\lesssim V(z)\lesssim U(z)$. Throughout the paper $\mu$ denotes the measure defined in \eqref{eq:mu}. We write $\Re z$ and $\Im z$ for the real and the imaginary part of $z\in\CC$, and $B_r(z)$ for the open disk of radius $r$ centered at $z$. Finally, $\Ker F$ denotes the zero set of an entire function $F$.\par
 The constants implicit in $\lesssim$, $\gtrsim$, $\asymp$, and $O(\cdot)$ may depend on the fixed parameter $\alpha$. Any other dependence is indicated explicitly.

\section{Preliminary results}\label{sect:Prelim}
For $u\in L^1(-\pi,\pi)$ we use the Fourier-transform convention
\begin{equation}\label{eq:fourier-convention}
\widehat{u}(z):=\int_{-\pi}^{\pi}u(t)e^{-itz}\,dt \quad \text{ for any } z\in\CC.
\end{equation}
For $w\in\K_{\alpha}$ define
\begin{equation}\label{eq:Fw-definition}
\F_w = \left\{ \widehat{u} \mid \int_{-\pi}^{\pi}\frac{|u(t)|^2}{w(t)}\,dt<\infty \right\}
\end{equation}
and equip this space with well-defined inner product $\left(\widehat{u},\widehat{v}\right)_w = \displaystyle \int_{-\pi}^{\pi}\frac{u(t)\overline{v(t)}}{w(t)}\,dt$.

\begin{proposition}\label{prop:Fw-rkhs}
The space $\F_w$ is a reproducing-kernel Hilbert space of entire functions of exponential type at most $\pi$. The operator $U_w:L^2(w)\longrightarrow\F_w$ defined by the rule $U_wf:=\widehat{fw}$ is unitary. For $z,\zeta\in\CC$ the reproducing kernel at $z$ is
\begin{equation}\label{eq:kernel-formula}
k_z^w(\zeta) = \widehat{w(t)e^{i\overline{z}t}}(\zeta) = \widehat{w}(\zeta-\overline{z}).
\end{equation}
Moreover, we have
\begin{equation}\label{eq:kernel-norm}
\|k_z^w\|_w^2 = k_z^w(z) = \int_{-\pi}^{\pi}w(t)e^{2t\Im z}\,dt.
\end{equation}
\end{proposition}
\begin{proof}
Since $w\in L^1(-\pi,\pi)$ we have $u\in L^1(-\pi,\pi)$ for every $\widehat u\in\F_w$ by the Cauchy--Schwarz inequality. Also $|\widehat u(z)|\leq \|u\|_{L^1}e^{\pi|\Im z|}$. Thus $\widehat u$ is entire of exponential type at most $\pi$. The definition of the norm shows that $U_w$ is unitary. We have
\[|\widehat u(z)| \leq\|\widehat u\|_w \left(\int_{-\pi}^{\pi}w(t)e^{2t\Im z}\,dt\right)^{1/2}, \]
so point evaluations are bounded. Finally, a direct calculation gives $(\widehat u,k_z^w)_w=\widehat u(z)$, and \eqref{eq:kernel-norm} follows by taking $\widehat u=k_z^w$.
\end{proof}

Proposition~\ref{prop:Fw-rkhs} gives $U_w\bigl(e^{i\lambda t}\bigr) = \widehat{w(t)e^{i\lambda t}} = k_{\lambda}^w$ for every $\lambda\in\R$. Being unitary, $U_w$ maps systems of exponentials with real frequencies to systems of reproducing kernels and preserves both exactness and hereditary completeness.

\begin{definition}
For $z\in\CC$ the \textit{normalized reproducing kernel} at $z$ is $\mathbb{K}_z^w=\frac{k_z^w}{\|k_z^w\|_w}$.
\end{definition}

\begin{lemma}\label{lem:fourier-norm-symmetries}
For $a\in\R$ the operator $(T_aF)(z) =F(z-a)$ is unitary on $(\F_w,\|\cdot\|_w)$. We have
\begin{equation}\label{eq:translation-kernels-fourier}
T_a k_z^w=k_{z+a}^w \quad \text{ and } \quad T_a\mathbb{K}_z^w=\mathbb{K}_{z+a}^w.
\end{equation}
In particular, for any $\lambda \in \R$ we have \begin{equation}\label{eq:real-kernel-norm-fourier}
\|k_{\lambda}^w\|_w^2 = \int_{-\pi}^{\pi}w(t)\,dt.
\end{equation}
Finally, for $w=w_{\alpha}$ the map $(JF)(z)=\overline{F(\overline{z})}$ is a conjugate-linear isometric involution on $\F_{w_{\alpha}}$.
\end{lemma}

\begin{proof}
The assertions follow directly from \eqref{eq:fourier-convention} and \eqref{eq:kernel-formula}: indeed, $T_a\widehat u=\widehat{u e^{iat}}$ and $J\widehat u=\widehat{\overline{u(-\,\cdot\,)}}$, and the latter map is isometric because $w$ is even.
\end{proof}

Recall that a space $H$ has the \textit{division property} if $F/(\,\cdot-\lambda)\in H$ whenever $F\in H$ and $F(\lambda)=0$. The spaces $\F_w$ have this property. We use the following standard generating-function description, see \cite[Section~1.2 and the beginning of Section~2]{BB11}.

\begin{proposition}\label{prop:generating-function}
Let $H$ be a reproducing-kernel Hilbert space of entire functions with the division property. Let $G$ be a generating function of the system $\{k_\lambda\}_{\lambda\in\Lambda}$, that is an entire function whose zero set is exactly $\Lambda$, all zeros are simple, and $G/(\,\cdot-\lambda)\in H$ for every $\lambda\in\Lambda$. Then the system
\[
\left\{ \frac{\|k_\lambda\|_H G(z)}{G'(\lambda)(z-\lambda)} \right\}_{\lambda\in\Lambda}
\]
is biorthogonal to the normalized kernels $\{k_\lambda/\|k_\lambda\|_H\}_{\lambda\in\Lambda}$.

Conversely, let $G$ be an entire function whose zeros are simple and form the set $\Lambda$, and assume that $G/(\,\cdot-\lambda)\in H$ for every $\lambda\in\Lambda$. Then the system $\{k_\lambda\}_{\lambda\in\Lambda}$ is minimal, and it is complete if and only if $TG\notin H$ for every entire function $T\not\equiv0$.
\end{proposition}

\begin{remark}\label{rem:ker}
In what follows we use the unnormalized quotients $g_\lambda(z)=G(z)/(z-\lambda)$. They differ from the biorthogonal
functions only by nonzero scalar factors, which do not affect closed linear spans or orthogonality relations.
\end{remark}

\subsection{The structure of $\F_w$}

We now restrict ourselves to the weight $w_{\alpha}(t)=(\pi-|t|)^{\alpha}$. Define the measure
\begin{equation}\label{eq:mu}
d\mu_{\alpha}(x+iy) =|y|^{\alpha-1}e^{-2\pi|y|}\,dx\,dy.
\end{equation}

\begin{proposition}\label{prop:model-characterization}
An entire function $F$ belongs to $\F_{w_{\alpha}}$ if and only if $F$ is of exponential type at most $\pi$ and $\displaystyle \int_{\CC}|F(z)|^2\,d\mu_{\alpha}(z)<\infty$. More precisely, for every $F\in\F_{w_{\alpha}}$ we have
\begin{equation}\label{eq:2}
\frac{2\pi\Gamma(\alpha)}{2^{\alpha}} \|F\|_{w_{\alpha}}^2 \leq \int_{\CC}|F(z)|^2\,d\mu_{\alpha}(z) \leq \frac{4\pi\Gamma(\alpha)}{2^{\alpha}} \|F\|_{w_{\alpha}}^2.
\end{equation}
\end{proposition}
\begin{proof}
Let $F=\widehat u\in\F_{w_\alpha}$. Plancherel's theorem gives
\begin{equation}\label{eq:1}
\int_{\CC}|F(z)|^2\,d\mu_\alpha(z)
=
2\pi\int_{-\pi}^{\pi}|u(t)|^2M_\alpha(t)\,dt,
\end{equation}
where
\[
\begin{aligned}
M_\alpha(t) = \int_{\R}e^{2ty}|y|^{\alpha-1}e^{-2\pi|y|}\,dy = \frac{\Gamma(\alpha)}{2^\alpha}
\left((\pi-t)^{-\alpha}+(\pi+t)^{-\alpha}\right).
\end{aligned}
\]
Hence we have
\[
\frac{\Gamma(\alpha)}{2^\alpha w_\alpha(t)}
\leq M_\alpha(t)
\leq
\frac{2\Gamma(\alpha)}{2^\alpha w_\alpha(t)}.
\]

Conversely, let $F$ be entire of exponential type at most $\pi$ and belong to $L^2(\mu_\alpha)$. By Fubini's theorem there exists $y_0\in(1,2)$ such that $F(\,\cdot+iy_0)\in L^2(\R)$. Now by the Paley--Wiener theorem applied to $G(z):=F(z+iy_0)$ we obtain $G(z)=\displaystyle \int_{-\pi}^{\pi}v(t)e^{-itz}\,dt$ for some $v\in L^2(-\pi,\pi)$. Then
\[
F(z)=G(z-iy_0)
=\int_{-\pi}^{\pi}v(t)e^{-ty_0}e^{-itz}\,dt
=\widehat{v(t)e^{-ty_0}}.
\]
Set $u(t)=v(t)e^{-ty_0}$. Repeating the computation above and using the lower bound for $M_\alpha$, we obtain
\[
\int_{\CC}|F(z)|^2\,d\mu_\alpha(z)
\geq
\frac{2\pi\Gamma(\alpha)}{2^\alpha}
\int_{-\pi}^{\pi}\frac{|u(t)|^2}{w_\alpha(t)}\,dt.
\]
Thus $F\in\F_{w_\alpha}$.
\end{proof}

\subsection{Properties of the measure}\label{subs:mu}

Throughout the subsection we write $\mu=\mu_\alpha$ and use the notation $\|F\|_\mu^2=\displaystyle \int_{\CC}|F(z)|^2\,d\mu_\alpha(z)$. By Proposition~\ref{prop:model-characterization}, this norm is equivalent to $\|\cdot\|_{w_\alpha}$. We retain the latter notation for the Fourier norm until a change of norm is explicitly announced.

The density in \eqref{eq:mu} is locally integrable because $\alpha>0$. The measure is invariant under horizontal translations and conjugation. Consequently, the operators 
\[
(T_aF)(z):=F(z-a),
\qquad
(JF)(z):=\overline{F(\overline z)}
\]
are unitary and conjugate-linearly isometric for $\|\cdot\|_\mu$ respectively. If $k_z^\mu$ denotes the reproducing kernel for this norm, then $T_a k_z^\mu=k_{z+a}^\mu$ and $\|k_x^\mu\|_\mu=\|k_0^\mu\|_\mu$ for any $a,x\in\R$.

For $z_0\in\CC$ and $r>0$ we have
\begin{equation}\label{eq:mass}
\mu_\alpha(B_r(z_0)) \leq 2r\int_{\R}|y|^{\alpha-1}e^{-2\pi|y|}\,dy \lesssim r.
\end{equation}
Moreover, for $x \in \R$ a horizontal translation and scaling give
\begin{equation}\label{eq:mass2}
\int_{B_r(x)}e^{2\pi|\Im z|}\,d\mu_\alpha(z) =C_\alpha r^{1+\alpha}, \quad \text{where} \quad C_\alpha=\int_{B_1(0)}|\Im z|^{\alpha-1}\,dx\,dy.
\end{equation}

Evaluating the integral in \eqref{eq:kernel-norm}, where the main contribution comes from neighbourhoods of $\pm\pi$, we obtain $\|k_z^{w_\alpha}\|_{w_\alpha}^2 \asymp \frac{e^{2\pi|\Im z|}}{(1+|\Im z|)^{1+\alpha}}$. Hence
\begin{equation}\label{eq:point}
|F(z)|
\lesssim
\|F\|_{w_\alpha}
\frac{e^{\pi|\Im z|}}{(1+|\Im z|)^{(1+\alpha)/2}}
\quad \text{ for any } F\in\F_{w_\alpha} \text{ and } z\in\CC.
\end{equation}
The same estimate holds with $\|F\|_\mu$ on the right-hand side by Proposition~\ref{prop:model-characterization}.

Finally, let $y>r>0$ and let $x\in\R$ be arbitrary. On every disk $B_r(x\pm i y)$ the density of $\mu_\alpha$ is bounded above and below by positive constants depending only on $\alpha$, $y$, and $r$. Since $|F|^2$ is subharmonic, the sub-mean value inequality gives
\begin{equation}\label{eq:weighted-submean}
|F(x\pm iy)|^2 \leq \frac{(y+r)^{1-\alpha}e^{2\pi(y+r)}}{\pi r^2} \int_{B_r(x\pm iy)}|F(z)|^2\,d\mu_\alpha(z)
\end{equation}
for every entire function $F$.

\subsection{Equivalent norms}\label{sect:eqiv}
Despite the following Proposition being folklore, we did not find it in the literature. So we present the proof for the reader's convenience.

\begin{proposition}\label{prop:equivalent-rkhs-norms}
Let $H$ be a reproducing-kernel Hilbert space of entire functions with respect to each of two equivalent norms $\|\cdot\|_0$ and $\|\cdot\|_1$, and let $\{k_z^{(j)}\}$ be the corresponding systems of reproducing kernels for $j=0,1$. Then for every $\Lambda \subset \CC$ the system $\{k_\lambda^{(0)}\}_{\lambda\in\Lambda}$ is hereditarily complete in $(H,\|\cdot\|_0)$ if and only if $\{k_\lambda^{(1)}\}_{\lambda\in\Lambda}$ is hereditarily complete in $(H,\|\cdot\|_1)$.
\end{proposition}

\begin{proof}
Completeness is equivalent to $\Lambda$ being a uniqueness set for $H$. Minimality is equivalent to the existence of functions $g_\lambda\in H$ for $\lambda\in\Lambda$ such that $g_\lambda(\lambda')=\delta_{\lambda\lambda'}$ for all $\lambda'\in\Lambda$. Thus exactness is independent of the norm. If the kernel systems are exact, these functions are unique and $(g_\lambda,k_\lambda'^{(j)})_j=g_\lambda(\lambda')=\delta_{\lambda\lambda'}$.
Hence $\{g_\lambda\}$ is biorthogonal to the unnormalized kernels for both norms, while the system biorthogonal to the normalized kernels is $\{\|k_\lambda^{(j)}\|_j g_\lambda\}$. These scalar factors do not affect completeness of mixed systems.

Fix a partition $\Lambda=\Lambda_1\sqcup\Lambda_2$ and put
\[
E_1:=\overline{\operatorname{span}} \{g_\lambda \mid \lambda\in\Lambda_1\}, \qquad N_2:=\{F\in H \mid F|_{\Lambda_2}=0\}.
\]
Since $E_1\subset N_2$, the mixed system $\{g_\lambda \mid \lambda\in\Lambda_1\} \cup\{k_\lambda^{(j)}\mid \lambda\in\Lambda_2\}$ is complete if and only if $N_2\cap E_1^{\perp_j}=\{0\}$. Equivalently, $E_1=N_2$. The subspace $N_2$ depends only on pointwise vanishing, and equivalent norms give the same closure in the definition of $E_1$. Hence this criterion is independent of $j$.
\end{proof}

\begin{corollary}\label{cor:comparable-weights}
For every $w\in\K_\alpha$ we have $\F_w=\F_{w_\alpha}$ as sets of entire functions. Moreover, for every $F\in\F_{w_\alpha}$ we have
\begin{equation}\label{eq:comparable-Fw-norms}
\frac{1}{C_w}\|F\|_{w_\alpha}^2 \leq \|F\|_w^2 \leq \frac{1}{c_w}\|F\|_{w_\alpha}^2.
\end{equation}
\end{corollary}

Hence, if the system $\{k^{w_\alpha}_\lambda\}_{\lambda\in\Lambda_\alpha}$ is exact but not hereditarily complete in $\F_{w_\alpha}$, then the same is true for $\{k^{w}_\lambda\}_{\lambda\in\Lambda_\alpha}$ in $\F_w$ for every $w\in\K_\alpha$. By Proposition~\ref{prop:Fw-rkhs}, the set $\Lambda_\alpha$ then yields a system of exponentials which is exact, but not hereditarily complete in $L^2(w)$ for every $w\in\K_\alpha$.

\subsection{The Bessel generator}
Put $W(t):=(\pi^2-t^2)^\alpha$ and $\Theta(z):=\widehat W(z)= \displaystyle \int_{-\pi}^{\pi}W(t)e^{-itz}\,dt$. Since $W\asymp w_\alpha$ on $(-\pi,\pi)$, we have $\Theta\in\F_{w_\alpha}$. The function $\Theta$ is even and entire of exponential type $\pi$.

 Let $\beta=\alpha+1/2$.  From Poisson's integral formula we have  $\Theta(z)=c_\alpha z^{-\beta}J_\beta(\pi z)$ (see \cite[(10.9.4)]{NIST}) for some constant $c_{\alpha}$, where $J_\beta$ denotes the Bessel function of the first kind of order $\beta$. Here $z^{-\beta}J_\beta(\pi z)$ is understood as its even power series. Let $j_{\beta,k}$ denote the $k$-th positive zero of $J_\beta$. The zeros of $\Theta$ are precisely the points $\pm j_{\beta,k}/\pi$ for $k\ge1$ and they are all real and simple (see \cite[(10.21.15)]{NIST}). Moreover, $j_{\beta,k}=(k+\beta/2-1/4)\pi+O(k^{-1})$ as $k\to\infty$ (see \cite[(10.21.19)]{NIST}). For every sufficiently small $d>0$
\begin{equation}\label{eq:theta}
|\Theta(z)|\asymp_{\alpha,d} \frac{e^{\pi|\Im z|}}{(1+|z|)^{1+\alpha}}, \qquad \operatorname{dist}(z,\Ker\Theta)\geq d,
\end{equation}
while the upper estimate holds for every $z\in\CC$. Both estimates follow from Hankel's expansion and its error bounds (see \cite[(10.17.3) and Section~10.17(iv)]{NIST}) and from the asymptotics of $j_{\beta,k}$ stated above.\par

\begin{lemma}\label{lem:theta}
We have $|\Theta'(r)|\asymp (1+|r|)^{-1-\alpha}$ for every zero $r$ of $\Theta$. For such $r$, the function 
$G_r(z) ={{\Theta(z)} \over {\Theta'(r)(z-r) }} $ lies in $\F_{w_\alpha}$ and satisfies $G_r(r)=1$ and $G_r(x)=0$ for every $x \in \Ker\Theta \setminus \{r\}$. Furthermore,
\begin{equation}\label{eq:Gnorm}
\|G_r\|_{w_\alpha}\lesssim (1+|r|)^\alpha.
\end{equation}
\end{lemma}

\begin{proof}
By evenness, it suffices to estimate $\Theta'$ at positive zeros. If $r>0$ is a zero of $\Theta$, then $\Theta'(r)=\pi c_\alpha r^{-\beta}J_\beta'(\pi r)$. Then \cite[(10.17.9) and (10.21.19)]{NIST} give the asserted estimate. Fix now an arbitrary zero $r$ of $\Theta$ and put
$$V_r(t)=\int_{-\pi}^{t}e^{-irs}W(s)\,ds \quad \text{and} \quad v_r(t)=i e^{irt}V_r(t).$$
Since $V_r(-\pi)=0$ and $V_r(\pi)=\Theta(r)=0$, we have $v_r(\pm\pi)=0$. Also $v_r'=irv_r+iW$ and  $\widehat{v_r'}(z)= iz\widehat v_r(z)$. Thus $(z-r)\widehat v_r(z)=\Theta(z)$. Since $W$ is absolutely continuous on $(-\pi,\pi)$ with $W'\in L^1(-\pi,\pi)$, we obtain
$$|V_r(t)|\leq {{W(t)+\|W'\|_{L^1(-\pi,\pi)}} \over {|r|}} \lesssim {{1} \over {1+|r|}}.$$
Now $\Theta/(\,\cdot-r)\in\F_{w_\alpha}$ and $\|\Theta/(\,\cdot-r)\|_{w_\alpha}\lesssim (1+|r|)^{-1}$ by $1/w_{\alpha} \in L^1(-\pi,\pi)$. To prove \eqref{eq:Gnorm} it remains to divide by $\Theta'(r)$.
\end{proof}

In the construction below, put $\Gamma=\Ker\Theta$.

\section{Proof of Theorem \ref{thm:main}}\label{sect:Proof}
The construction follows the scheme used in the proof of \cite[Theorem~2.1]{BBB}.

\subsection{The reformulation}
Fix $0<\alpha<1$. Throughout this section and Section~\ref{sect:Tech}, we write $w=w_\alpha$ and work only in the space $\F_{w_{\alpha}}$. By the argument in Subsection \ref{sect:eqiv}, it suffices to construct a set $\Lambda_\alpha\subset\R$ whose reproducing kernels are complete and minimal but not hereditarily complete in this space, equipped with any one equivalent Hilbert norm. The same set will then work for every weight in $\K_\alpha$.

\subsection{Change of norms}
From now on the symbols $(\cdot,\cdot)_w$ and $\|\cdot\|_w$ denote the inner product $(F,H)_w=\displaystyle \int_{\CC}F(z)\overline{H(z)}\,d\mu_\alpha(z)$ and norm $ \|F\|_w^2 =(F,F)_w$. We put $\mu=\mu_\alpha$. Equivalence with the Fourier norm of Section~\ref{sect:Prelim} follows from Proposition~\ref{prop:model-characterization}. 

Write $k_z:=k_z^\mu$ and $\mathbb K_z:=k_z/\|k_z\|_w$ for the reproducing kernels and normalized reproducing kernels of this norm. By Subsection~\ref{subs:mu}, $\|k_x\|_w=\|k_0\|_w$ for $x\in\R$. The conjugation isometry also gives $Jk_x=k_x$ for real $x$. Thus these kernels are real on $\R$.

For an exact kernel system and a partition $\Lambda=\Lambda_1\sqcup\Lambda_2$ let $g_\lambda$ be the functions defined in Remark~\ref{rem:ker}. Now the criterion in the proof of Proposition~\ref{prop:equivalent-rkhs-norms} shows that it suffices to find $F,H\in\F_w$ such that
\[\begin{cases}
F(\lambda)=0, & \lambda \in \Lambda_2 \\
(H,g_\lambda)_w=0, & \lambda\in\Lambda_1\\
(F,H)_w\ne0.
\end{cases}
\]

\subsection{Construction of the system}

We construct $F=\Theta R$, where $R$ has positive residues, and then obtain the desired set by replacing selected zeros of $R$ by nearby points forming a lacunary sequence. Note that constants in this section may depend on $\alpha$.

Set $a=(1+\alpha)/2$ and $b=(1-\alpha)/4$. Then $\alpha<a<1$, $a+b<1$ and $\alpha-a=a-1=-2b$. Let $\gamma_0<\gamma_1$ be the first two positive zeros of $\Theta$. For $\gamma\in\Gamma$ put $\varepsilon_\gamma:=\operatorname{sign}\Theta'(\gamma)$ and define
\[
\widetilde G_\gamma(z)=\varepsilon_\gamma G_\gamma(z)
=\frac{\Theta(z)}{|\Theta'(\gamma)|(z-\gamma)}.
\]
These functions are real on $\R$, vanish on $\Gamma\setminus\{\gamma\}$, and satisfy $\widetilde G_\gamma(\gamma)=\varepsilon_\gamma$.

For a sufficiently large $Q$ set $u_n:=2^{n-1}Q$. Let $p_n$ be the smallest point of $\Gamma$ not less than $u_n$, and let $q_n$ be its successor. The zero asymptotics in Section~\ref{sect:Prelim} give $p_n=2^{n-1}Q+O(1)$ and $q_n-p_n\asymp1$ uniformly in $n$. Observe that $p_{n+1}\ge \tfrac32p_n$. We take $Q$ large enough that $\gamma_1<p_1$. For $c\in[1,2]$, put $\sigma=c(\widetilde G_{\gamma_0}+\widetilde G_{\gamma_1})$ and $\mathcal P=\{\gamma_0,\gamma_1\}\cup\{p_n,q_n\}_{n\ge1}$. The constant $c$ will be chosen later. Define
\begin{equation}\label{eq:F}
F = \sigma + \sum_{n=1}^{\infty} {{\widetilde{G}_{p_n} + \widetilde{G}_{q_n}} \over {p_n^{a}}}.
\end{equation}

\begin{lemma}\label{lem:real}
The following statements hold:
\begin{enumerate}
\item $F \in \F_w$ and $\|F - \sigma\|_w \lesssim Q^{\alpha - a}$, while $\|\sigma\|_w \asymp 1$ uniformly in $Q$ and in $c \in [1,2]$;
\item $F(x) \in \R$ for any $x \in \R$;
\item $F(\gamma) = 0$ for any $\gamma \in \Gamma \setminus \mathcal{P}$, and for any $n \ge 1$ we have
$$F(p_n) = {{\varepsilon_{p_n}} \over {p_n^{a}}} \quad \text{and} \quad F(q_n) = {{\varepsilon_{q_n}} \over {p_n^{a}}}.$$
\end{enumerate}
\end{lemma}
\begin{proof}
Using lacunarity and \eqref{eq:Gnorm} we obtain
\[
\|F-\sigma\|_w\lesssim\sum_{n\ge1}p_n^{\alpha-a} \lesssim Q^{\alpha-a}.
\]
Thus \eqref{eq:F} converges absolutely in $\F_w$. Since $\sigma/c$ is fixed and nonzero, $\|\sigma\|_w\asymp1$ for $c\in[1,2]$. The reality and interpolation assertions follow termwise from the corresponding properties of $\widetilde G_\gamma$.
\end{proof}

Put $b_{\gamma_j} ={{c} \over {|\Theta'(\gamma_j)|}}$ for $j=0,1$ and put $b_{p_n}={{1} \over {p_n^a|\Theta'(p_n)|}}$. Put $b_{q_n} = {{1} \over {p_n^a|\Theta'(q_n)|}}$. Define
\[
R(z)=\sum_{\gamma\in\mathcal P}\frac{b_\gamma}{z-\gamma}.
\]
All residues of $R$ are positive. Lemma~\ref{lem:theta} gives $b_{p_n}\asymp b_{q_n}\asymp p_n^a$. Hence the series converges locally uniformly off $\mathcal P$.

Observe that 
$$\Im R(x+iy)=-y\sum_{\gamma\in\mathcal P}b_\gamma|x+iy-\gamma|^{-2} \quad \text{ for any } y \ne 0.$$
Hence, all zeros of $R$ are real. Now we have $R'(x)=-\sum_{\gamma\in\mathcal P}b_\gamma(x-\gamma)^{-2}<0$ for all $x \in \R\setminus\mathcal P$. Thus $R$ has one simple zero in each bounded gap between consecutive poles, where it decreases from $+\infty$ to $-\infty$, and no zeros to the left of $\gamma_0$. Denote its zero in each interval $(p_n,q_n)$ by $x_n^*$.

For sufficiently large $Q$ and all $z\not= p_n,q_n$ such that $|z-p_n|\leq 2p_n^a$ we have
\begin{equation}\label{eq:Rloc}
R(z)=\frac{b_{p_n}}{z-p_n}+\frac{b_{q_n}}{z-q_n} +O\left(p_n^{-1}+p_n^{\alpha-a}\right).
\end{equation}
Indeed, the fixed poles contribute $O(p_n^{-1})$. The poles $p_m,q_m$ with $m<n$ are at distance at least $p_n/6$ from this disk, and those with $m>n$ are at distance at least $p_m/4$. Their contributions are bounded by constant multiples of
\[
\frac1{p_n}\sum_{m<n}p_m^a \leq\frac{p_n^{a-1}}{(3/2)^a-1} \quad\text{and}\quad \sum_{m>n}p_m^{a-1} \leq\frac{p_n^{a-1}}{(3/2)^{1-a}-1}.
\]
Here $a-1=\alpha-a$, and the estimates are uniform in $n$, $Q$, and $c\in[1,2]$.

Observe that $q_n-p_n\asymp1$ and $b_{p_n}\asymp b_{q_n}\asymp p_n^a$. Hence we can choose $c_0>0$ smaller than half the minimum gap so that the sum of the two principal terms in \eqref{eq:Rloc} is at least $p_n^a$ at $p_n+c_0$ and at most $-p_n^a$ at $q_n-c_0$. The remainder is $O(p_n^{a-1})$, so $R(p_n+c_0)>0>R(q_n-c_0)$ for sufficiently large $Q$ uniformly in $n,c$. Monotonicity gives $\min(x_n^*-p_n,q_n-x_n^*)\ge c_0$.

\begin{lemma}\label{lem:herglotz}
We have $F=\Theta R$. For any sufficiently large $Q$ there is a countable set $E_Q\subset[1,2]$ such that for $c\in[1,2]\setminus E_Q$ one has $\Ker R\cap\Gamma=\emptyset$. For such $c$ the function $F$ has only simple real zeros and 
\[
\Ker F=(\Gamma\setminus\mathcal P)\cup\Ker R.
\]
\end{lemma}
\begin{proof}
The identity follows from \eqref{eq:F}. At each pole $\gamma\in\mathcal P$, the singularity of $\Theta R$ is removable and $F(\gamma)=b_\gamma\Theta'(\gamma)\ne0$.

Write $R_c=cA+B$, where $B$ is independent of $c$ and 
$$A(x)= {{1} \over {|\Theta'(\gamma_0)|(x-\gamma_0)}} + {{1} \over {|\Theta'(\gamma_1)|(x-\gamma_1)}}.$$
For each $\gamma\in\Gamma\setminus\mathcal P$ with $\gamma>\gamma_1$ we obtain $A(\gamma)>0$. Hence, $R_c(\gamma)=0$ excludes at most one value of $c$. There are no zeros of $R$ to the left of $\gamma_0$, and $(\gamma_0,\gamma_1)$ contains no point of $\Gamma$. Thus the exceptional set $E_Q$ is countable. For $c\in[1,2]\setminus E_Q$ the simple zeros of $\Theta$ and $R$ are disjoint, which proves the assertion.
\end{proof}

For $c\notin E_Q$, put $\Lambda_2=\Ker F\setminus\{x_n^*\}_{n\ge1}$. The following selection is uniform in $Q$ and $c$.

\begin{lemma}\label{lem:selection}
There exists $\delta\in(0,1/4)$ depending only on $\alpha$ such that for every sufficiently large $Q$ and every $c\in[1,2]$ one can choose $v_n \in \R$ with
\[
|v_n-(p_n-p_n^a)|<1 \quad \text{and} \quad \operatorname{dist}(v_n,\Gamma\cup\Ker R)\ge\delta.
\]
The number $\delta$ may also be chosen small enough for \eqref{eq:theta}.
\end{lemma}
\begin{proof}
Let $h>0$ be a separation constant for $\Gamma$. Put $M=\lceil4/h\rceil+1$ and $N=2M+1$. Every interval of length $4$ contains at most $M$ points of $\Gamma$, hence at most $M$ poles of $R$. Between any two zeros of $R$ there is a pole, so such an interval contains at most $M+1$ zeros of $R$. Choose $0<\delta<\min(1/4,1/2N)$ small enough for \eqref{eq:theta}. In the interval of length $2$ centered at $p_n-p_n^a$, the $\delta$-neighborhoods of the points of $\Gamma\cup\Ker R$ in the concentric interval of length $4$ have total length at most $2N\delta<1$. A point outside these neighborhoods has the required distance.
\end{proof}

Choose such points $v_n$ and put $\Lambda_1=\{v_n\}_{n\ge1}$. In particular, $\operatorname{dist}(v_n,\Lambda_2)\ge\delta$ and $v_n>1$ for large $Q$. Define the functions
\[
S(z):=\prod_{n\ge1}\left(1-\frac{z}{x_n^*}\right),
\qquad
G_1(z):=\prod_{n\ge1}\left(1-\frac{z}{v_n}\right),
\qquad G_2:=\frac{F}{S}
\]
and put $G=G_1G_2$ and $\Lambda=\Lambda_1\sqcup\Lambda_2$. The products converge locally uniformly by lacunarity, so $G$ is entire with simple zero set $\Lambda$. The exactness of the kernel system will be proved in Section~\ref{sect:Tech}. So there exists $g_\lambda(z)=G(z)/(z-\lambda) \in \F_w$ as in Remark~\ref{rem:ker}.\par
For real numbers $d_n\in(-1,1)$ define
\begin{equation} \label{eq:H}
H(z) = \sigma(z) + \sum_{n \ge 1} {{d_n} \over {p_n^{b}}} \mathbb{K}_{p_n} (z).
\end{equation}
Note that for every real sequence $\{d_n\}$ such that $|d_n|<1$ the series converges absolutely in $\F_w$. Its terms are real on $\R$, hence so is $H$. Also we have $\|H-\sigma\|_w\leq\sum_n p_n^{-b}\lesssim Q^{-b}$.  The coefficients $d_n$ will be chosen in Section~\ref{sect:Tech} to ensure $(H,g_{v_n})_w=0$ for every $n$.

The final order of choices is $\alpha$, then $\delta$ from Lemma~\ref{lem:selection}, then $Q$, then $c\in[1,2]\setminus E_Q$, then $\{v_n\}$, and finally $\{d_n\}$. Once $\alpha$ and $\delta$ are fixed, estimates in Section~\ref{sect:Tech} may depend on them, but any additional dependence on $Q$ will be indicated explicitly.

\subsection{Properties of the system}

In Section \ref{sect:Tech} we formulate and prove technical lemmas to show that there exist parameters $Q$, $c$ and $d_n \in (-1,1)$ such that the system $\{\mathbb{K}_{\lambda}\}_{\lambda \in \Lambda}$ is complete and minimal in $\F_w$, while the functions $F$ and $H$ satisfy the following relations:
$$(F, \mathbb{K}_{\lambda})_w = 0 \text{ for any } \lambda \in \Lambda_2, \quad (H, g_{\lambda})_w = 0 \text{ for any } \lambda \in \Lambda_1.$$
Finally, we show that one can choose the parameters such that $(F,H)_w \neq 0$. Thus the corresponding mixed system is not complete in $\F_w$ by the criterion in the proof of Proposition~\ref{prop:equivalent-rkhs-norms}.

\section{Technical lemmas}\label{sect:Tech}
We continue with $w=w_\alpha$, $\mu=\mu_\alpha$, and the area norm introduced in Section~\ref{sect:Proof}. Recall that $\|F\|_w^2=\displaystyle \int_{\CC}|F|^2\,d\mu_\alpha$. Put $\Omega_n=B_{2p_n^a}(p_n)$ for $n\ge1$ and $\Omega_0=\CC\setminus\bigcup_{n\ge1}\Omega_n$. Note that the intervals $[p_n-2p_n^a,p_n+2p_n^a]$ have gaps at least $1$  for sufficiently large $Q$ and $q_n,x_n^*,v_n\in\Omega_n$. We use $1+\alpha=2a$ and $b_{p_n}\asymp b_{q_n}\asymp p_n^a$.

\begin{lemma}\label{lem:below}
For every $z\in\CC$ with $|\Im z|\ge\delta$ we have
\[
|F(z)|\gtrsim Q^{a-2}\frac{e^{\pi|\Im z|}}{(1+|z|)^{2+a}}.
\]
\end{lemma}
\begin{proof}
Put $r=\max(|z|,Q)$ and choose the first index $n$ with $p_n\ge|z|$. Then $p_n\asymp r$ and $|z-p_n|\lesssim r$. Positivity of the residues gives
\[
|R(z)|\ge|\Im R(z)|
\ge\frac{\delta b_{p_n}}{|z-p_n|^2}
\gtrsim r^{a-2}.
\]
Since $\Gamma\subset\R$, one can multiply by the lower bound in \eqref{eq:theta} and use $r\le Q(1+|z|)$.
\end{proof}

\begin{lemma}\label{lem:tech}
The following estimates hold:
\begin{enumerate}
\item $|G_1(z)/S(z)|\asymp1$ on $\Omega_0$ and $\left|{{G_1(z)} \over {S(z)}}\right|\asymp \left|{{z-v_n} \over {z-x_n^*}} \right|$ on $\Omega_n\setminus\{x_n^*\}$;
\item $|G(z)|\lesssim e^{\pi|\Im z|}$ for every $z\in\CC$;
\item for $|\Im z|\ge\delta$ we have $|G(z)|\gtrsim Q^{a-2}e^{\pi|\Im z|}(1+|z|)^{-2-2a}$.
\end{enumerate}
\end{lemma}
\begin{proof}
Note that $x_m^*-v_m >0$ for large $Q$. Hence we obtain
\[
\frac{G_1(z)}{S(z)}
=\left(\prod_m\frac{x_m^*}{v_m}\right)\prod_m(1+e_m(z)) \quad \text{where } e_m(z)=\frac{v_m-x_m^*}{x_m^*-z}.
\]
The first product lies between $1$ and $\exp(CQ^{a-1})$, since $\sum_m {{x_m^*-v_m} \over {v_m}} \lesssim\sum_m p_m^{a-1}\lesssim Q^{a-1}$. On $\Omega_0$ we have 
$$|e_m(z)|\leq {{p_m^a+C} \over {2p_m^a-C}} \leq \frac{2}{3}$$
for sufficiently large $Q$. Put $r=\max(|z|,Q)$ and divide the indices into three groups according to whether $p_m\leq \frac{r}{2}$, $\frac{r}{2}<p_m<2r$, or $p_m\ge 2r$. The first and third groups contribute at most
\[
C\left(\frac1r\sum_{p_m\leq r/2}p_m^a +\sum_{p_m\ge2r}p_m^{a-1}\right)\lesssim r^{a-1},
\]
while the middle group has at most five indices. Hence $\sum_m|e_m(z)|\lesssim1$ on $\Omega_0$. For $z\in\Omega_n$ excluding the $n$-th factor gives
\[
\sum_{m\ne n}|e_m(z)|
\lesssim\frac1{p_n}\sum_{m<n}p_m^a+\sum_{m>n}p_m^{a-1}
\lesssim p_n^{a-1}.
\]
Thus $|e_m(z)|\leq 2/3$ for every remaining index $m$ when $Q$ is sufficiently large. The corresponding sums of logarithms converge absolutely and are uniformly bounded. Exponentiating yields both product bounds. By isolating $1+e_n(z)={{z-v_n} \over {z-x_n^*}}$ we prove (1).\par
On $\Omega_0$ one may use Lemma~\ref{lem:real} and \eqref{eq:point}. Together with (1) we obtain (2). \par
On $\Omega_n$ we prove the stronger estimate
\begin{equation}\label{eq:Glocal}
|G(z)|\lesssim e^{\pi|\Im z|} \left(\frac1{(1+|z-p_n|)^2} +\frac{p_n^{-1}}{1+|z-p_n|}\right).
\end{equation}
Choose $A$ so that $A/2-2$ exceeds a uniform upper bound for $q_n-p_n$. On $B_{3A}(p_n)$ one may use \eqref{eq:theta}, so
$$\left|{{\Theta(z)} \over {z-p_n}} \right| +\left|{{\Theta(z)} \over {z-q_n}} \right| \lesssim p_n^{-2a}.$$
Equation~\eqref{eq:Rloc} therefore gives $|F(z)|\lesssim p_n^{-a}$ there. Since $F(x_n^*)=0$, the maximum principle applied to $F(z)/(z-x_n^*)$ on $B_{2A}(x_n^*)$ gives the same bound for this quotient on $B_A(p_n)$. Part (1) yields $|G|\lesssim1$ on $B_A(p_n)$.\par

If $s=|z-p_n|\ge A$ and $z\in\Omega_n$, the distances to $q_n,x_n^*$ are comparable to $s$. By \eqref{eq:Rloc} and \eqref{eq:theta} one has
$$|F(z)|\lesssim e^{\pi|\Im z|}(p_n^{-a}s^{-1}+p_n^{-1-a}).$$
Now using (1) and $|z-v_n|\lesssim p_n^a$ gives $|G(z)|\lesssim e^{\pi|\Im z|}(s^{-2}+p_n^{-1}s^{-1})$. This proves \eqref{eq:Glocal} and (2).\par

Finally, combine Lemma~\ref{lem:below} with (1). On $\Omega_0$ the quotient is bounded below. On $\Omega_n$ with $|\Im z|\ge\delta$ we have $|z-v_n|\ge\delta$ and $|z-x_n^*| \lesssim p_n^a \asymp(1+|z|)^a$. These bounds give (3).
\end{proof}

\begin{lemma}\label{lem:in}
For every $\lambda\in\Lambda$ we have $g_\lambda\in\F_w$ and $\sup_{\lambda\in\Lambda}\|g_\lambda\|_w\lesssim1$.
\end{lemma}
\begin{proof}
Lemma~\ref{lem:tech} and the maximum principle give $|g_\lambda|\lesssim1$ on $B_1(\lambda)$. Outside this disk we have $|g_\lambda(z)|\lesssim e^{\pi|\Im z|}$. Hence, $g_\lambda$ has exponential type at most $\pi$. The integral over $B_1(\lambda)$ is bounded by \eqref{eq:mass}. On $\Omega_0\setminus B_1(\lambda)$, part (1) gives
$|g_\lambda|\lesssim|F|$.\par
Using \eqref{eq:Glocal} and \eqref{eq:mass2} one can deduce
\[
\int_{\Omega_n}|G|^2\,d\mu_\alpha \lesssim1+\sum_{k\ge0} \left(2^{k(\alpha-3)}+p_n^{-2}2^{k(\alpha-1)}\right) \lesssim1.
\]
Consequently,
\[
\int_{\Omega_n\setminus B_1(\lambda)}|g_\lambda|^2\,d\mu_\alpha \lesssim\frac1{1+\operatorname{dist}(\lambda,\Omega_n)^2}.
\]
The real projections of $\Omega_n$ have gaps at least $1$. Ordering these intervals on either side of $\lambda\in\R$ bounds the sum of the last quantities by $3+2\sum_{k\ge1}(1+k^2)^{-1}$. Thus the area integrals are uniformly bounded, and Proposition~\ref{prop:model-characterization} completes the proof.
\end{proof}

\begin{lemma}\label{lem:uneq}
For every $\lambda\in\Lambda_1$ we have $|(\sigma,g_\lambda)_w|\lesssim\lambda^{-1}$.
\end{lemma}
\begin{proof}
The definition of $\sigma$ and \eqref{eq:theta} give $|\sigma(z)|\lesssim e^{\pi|\Im z|}(1+|z|)^{-2-\alpha}$. By Lemma~\ref{lem:tech},
\[
\int_{\CC}|\sigma G|\,d\mu_\alpha \lesssim\int_{\CC}\frac{|y|^{\alpha-1}}{(1+|z|)^{2+\alpha}}\,dx\,dy <\infty.
\]
For $\lambda\in\Lambda_1$ put $B:=B_{\lambda/2}(\lambda)$. The integral of $|\sigma g_\lambda|$ over $\CC\setminus B$ is therefore $O(\lambda^{-1})$. On $B$ we have $|z|\ge\lambda/2$ and $\displaystyle \int_B|y|^{\alpha-1}\,dx\,dy\lesssim\lambda^{1+\alpha}$. Lemma~\ref{lem:in} gives
\[
\int_B|\sigma g_\lambda|\,d\mu_\alpha \lesssim\left(\lambda^{-4-2\alpha}\lambda^{1+\alpha}\right)^{1/2} =\lambda^{-(3+\alpha)/2}\lesssim\lambda^{-1}.
\]
\end{proof}

\subsection{The system $\{\mathbb{K}_{\lambda}\}_{\lambda \in \Lambda}$ is complete and minimal}\label{subs:compl}
The kernel system is minimal by Lemma~\ref{lem:in} and Proposition~\ref{prop:generating-function}. Its biorthogonal functions are $\|k_\lambda\|_w g_\lambda/G'(\lambda)$.

To prove completeness, let $L$ be entire and suppose $LG\in\F_w$. We need to show that $L\equiv0$. Put
$$m(r,f)=(2\pi)^{-1}\displaystyle \int_0^{2\pi}\log^+|f(re^{it})|\,dt \quad \text{and} \quad M(r,f)=\max_{|z|\leq r}|f(z)|.$$
Both $G$ and $LG$ have exponential type at most $\pi$. Note that $G(0)\ne 0$. Indeed, $S(0)=G_1(0)=1$, $\Theta(0)>0$, and
\[
R(0)=-\sum_{\gamma\in\mathcal P}\frac{b_\gamma}{\gamma}<0.
\]
So we have $G(0)=F(0)=\Theta(0)R(0)\ne0$. Jensen's formula now gives 
$$m(r,L)\le m(r,LG)+m(r,G)-\log|G(0)|.$$
 The estimate $\log^+M(r,L)\leq 3m(2r,L)$ therefore shows that $L$ has finite exponential type. The pointwise estimate and Lemma~\ref{lem:tech} give
$$|L(z)|\lesssim(1+|z|)^{2+2a} \quad \text{ for } |\Im z|\ge\delta.$$
On $|z|=r$ the remaining arcs have angular length $O(r^{-1})$. The finite exponential type of $L$ bounds their contribution to $m(r,L)$ by a constant. Hence $|m(r,L)| \lesssim \log r$ and $L$ is a polynomial.\par
Put $z_n=p_n+i$. Positivity of the residues gives $|R(z_n)|\ge|\Im R(z_n)|\ge b_{p_n}\asymp p_n^a$, whereas \eqref{eq:Rloc} gives $|R(z_n)|\lesssim p_n^a$. The Bessel estimate and the product comparison yield 
$$|\Theta(z_n)|\asymp p_n^{-2a} \quad \text{and} \quad |G_1(z_n)/S(z_n)|\asymp p_n^a.$$
Hence, $|G(z_n)|\asymp1$. The disks $B_{1/4}(z_n)$ are disjoint and the density of $\mu_\alpha$ is bounded above and below there. The weighted sub-mean estimate gives
\[
\|LG\|_w^2 \ge\sum_n\int_{B_{1/4}(z_n)}|LG|^2\,d\mu_\alpha \gtrsim\sum_n|L(z_n)G(z_n)|^2.
\]
The last series diverges for every nonzero polynomial $L$. Therefore $L\equiv0$.

\subsection{For any $\lambda \in \Lambda_2$ we have $(F, \mathbb{K}_{\lambda})_w =0$.}
It is clear that $\Lambda_2 \subset \Ker F$ by the construction of $\Lambda_2$. So $(F, \mathbb{K}_{\lambda})_w  = 0$ for any $\lambda \in \Lambda_2$.

\subsection{For any $\lambda \in \Lambda_1$ we have $(H, g_{\lambda})_w =0$.}
We choose the coefficients in \eqref{eq:H} so that $(H,g_{v_n})_w=0$ for every $n$. We have $|G(p_m)|\asymp1$ by Lemmas~\ref{lem:real} and~\ref{lem:tech}. For real $d_m$ the reproducing property turns these conditions into
\[
\sum_{m\ge1}\frac{G(p_m)d_m}{p_m^b(p_m-v_n)} =-\|k_0\|_w(g_{v_n},\sigma)_w.
\]
All coefficients are real by conjugation symmetry. Dividing each equation by its diagonal coefficient gives $(I+A)D=B$, where $D=(d_n)$ and $A$ has zero diagonal. By Lemma~\ref{lem:uneq} we have $|B_n|\lesssim p_n^{-b}$, since $a+2b=1$. Moreover, for any $m\ne n$ one can obtain
$$|p_n-v_n|\asymp p_n^a \quad \text{and} \quad |p_m-v_n|\gtrsim\max(p_m,p_n).$$
Hence $|A_{nm}|\lesssim {{p_n^{1-b}} \over {p_m^b\max(p_m,p_n)}}$. Using lacunarity we obtain
$$
\sum_{m<n}|A_{nm}| \lesssim p_n^{-b}\sum_{m<n}p_m^{-b} \lesssim Q^{-2b} \quad \text{and} \quad \sum_{m>n}|A_{nm}| \lesssim p_n^{1-b}\sum_{m>n}p_m^{-1-b} \lesssim p_n^{-2b}\le Q^{-2b}.
$$

Thus $\|A\|\lesssim Q^{-2b}$, where $\|A\|$ denotes the operator norm on $\ell^\infty$. We now have $\|B\|_\infty\lesssim Q^{-b}$ uniformly in $c$ and the admissible choices of $\{v_n\}$. Finally the Neumann series gives a real solution $D=(I+A)^{-1}B$ with $\|D\|_\infty\leq C Q^{-b}<1$ for sufficiently large $Q$.

\subsection{We have $(H, F)_w \ne 0$.}
It is sufficient to prove the following lemma.
\begin{lemma}\label{lem:ne0}There exists a constant $C>0$ independent of $Q$ such that
$$\left\|F - \sigma  \right\|_w  + \left\|H - \sigma \right\|_w  \leq C Q^{-{{1-\alpha} \over {4}}}.$$
\end{lemma}
\begin{proof} By Lemma \ref{lem:real}(1) we have $\|F - \sigma\|_w \lesssim Q^{\alpha-a} = Q^{-(1-\alpha)/2}$, while $\|\mathbb{K}_{\lambda}\|_w = 1$ and $|d_n|<1$ give
$$\left\|H - \sigma \right\|_w  = \left\| \sum_{n \ge 1} {{d_n} \over {p_n^{b}}} \mathbb{K}_{p_n}\right \|_w \leq \sum_{n=1}^{\infty} {{1} \over {p_n^{b}}} \lesssim {{1} \over {Q^{b}}} = {{1} \over {Q^{(1-\alpha)/4}}}. $$
\end{proof}
Now we have $(F,H)_w=\|\sigma\|_w^2+O(Q^{-b})$ with $\|\sigma\|_w\asymp1$ by Lemmas~\ref{lem:real} and \ref{lem:ne0}. Hence $(F,H)_w\ne0$ for sufficiently large $Q$. This completes the proof of Theorem \ref{thm:main}.\par\medskip

The author is deeply grateful to his scientific consultant Hakan Hedenmalm for invaluable discussions and guidance.

\end{document}